\documentclass[amstex,12pt, amssymb]{article}

\usepackage{mathtext}
\usepackage[cp1251]{inputenc}
\usepackage[T2A]{fontenc}
\usepackage[dvips]{graphicx}
\usepackage{amsmath}
\usepackage{amssymb}
\usepackage{amsxtra}
\usepackage{latexsym}
\usepackage{ifthen}

\newcounter{lemma}[section]

\newcounter{corollary}[section]

\newcounter{remark}[section]

\newcounter{theorem}[section]

\newcounter{proposition}[section]

\numberwithin{equation}{section}

\begin{document}

\markboth{\centerline{E.~SEVOST'YANOV, A.~MARKYSH}} {\centerline{ON
SOKHOTSKI--CASORATI--WEIERSTRASS THEOREM }}

\def\cc{\setcounter{equation}{0}
\setcounter{figure}{0}\setcounter{table}{0}}

\overfullrule=0pt


\author{{E.~SEVOST'YANOV, A.~MARKYSH}\\}

\title{
{\bf ON SOKHOTSKI--CASORATI--WEIERSTRASS THEOREM ON METRIC SPACES}}

\date{\today}
\maketitle

\begin{abstract} The paper is devoted to maps of metric spaces whose quasiconformal
characteristic satisfies certain restrictions of integral nature. We
prove that so-called ring $Q$-mappings have a continuous extension
to an isolated boundary point if the function  $Q(x)$ has finite
mean oscillation at this point. As a corollary, we obtain an analog
of the well-known Sokhotski–Weierstrass theorem on ring
$Q$-mappings.
\end{abstract}


\section{Introduction}

This paper is devoted to mappings with bounded and finite
distortion, which have been studied in recent time in a series of
papers by various authors, see, e.g., \cite{AS}, \cite{Cr$_1$},
\cite{GSS}, \cite{He}, \cite{MRV$_2$}, \cite{MRSY}, \cite{Ri} and
\cite{Sev$_1$}--\cite{Sev$_2$}. The main goal of the present paper
is to prove an analog of well-know Sokhotski–Weierstrass theorem in
metric spaces. The corresponding analogs for more general ring
$Q$-mappings in ${\Bbb R}^n$ were proved by the author in
\cite{Sev$_1$}--\cite{Sev$_3$} and, more later, by Cristea for some
another classes of mappings \cite{Cr$_1$}. Results concerning
removal of isolated singularities for mappings with bounded
distorsion (quasiregular mappings) have been obtained mostly in a
series of papers by Martio, Rickman and V\"{a}is\"{a}l\"{a}, see
\cite{MRV$_2$} and \cite{Ri}. Below we present the basic results
concerning removal of isolated singularities for ring $Q$-mappings
in metric spaces that fit roughly into the pattern of
\cite{Sev$_1$}.

\medskip
Everywhere further $(X, d, \mu)$ and $\left(X^{\,\prime},
d^{\,\prime}, \mu^{\,\prime}\right)$ are metric spaces with metrics
$d$ and $d^{\,\prime}$ and locally finite Borel measures $\mu$ and
$\mu^{\,\prime},$ correspondingly. A set $E$ is said to be {\it path
connected} if any two points $x_1$ and $x_2$ in $E$ can be joined by
a path $\gamma:[0, 1]\rightarrow E,$ $\gamma(0)=x_1$ and
$\gamma(1)=x_2.$ Given a metric space $(X, d, \mu)$ with a measure
$\mu,$ a {\it domain} in $X$ is an open path-connected set in $X.$
Similarly, we say that a domain $G$ is {\it locally path connected
(rectifiable)} at a point $x_0\in \partial G,$ if, for every
neighborhood $U$ of the point $x_0,$ there is a neighborhood
$V\subset U$ such that $V\cap G$ is path connected. Given a family
of paths $\Gamma$ in $X$, a Borel function
$\rho:X\rightarrow[0,\infty]$ is called {\it admissible} for
$\Gamma$, abbr. $\rho\in {\rm adm}\,\Gamma$, if
$\int\limits_{\gamma}\rho\,ds\ \geqslant\ 1$ for all (locally
rectifiable) $\gamma\in\Gamma$.
\medskip
Everywhere further, for any sets $E, F,$ and $G$ in $X$, we denote
by $\Gamma(E, F, G)$ the family of all continuous curves $\gamma:[0,
1]\rightarrow X$ such that $\gamma(0)\in E,$ $\gamma(1)\in F,$ and
$\gamma(t)\in G$ for all $t\in (0, 1).$ For $x_0\in X$ and $r>0,$
the ball $\{x\in X: d(x, x_0)<r\}$ is denoted by $B(x_0, r),$ and
the sphere $\{x\in X: d(x, x_0)=r\}$ is denoted by $S(x_0, r).$

\medskip
An open set any two points of which can be connected by a curve is
called a domain in $X.$ Given $p\geqslant 1,$ the $p$-modulus of the
family $\Gamma$ is the number
\begin{equation}\label{eq13.5}
M_p(\Gamma)=\inf\limits_{\rho\in {\rm
adm}\,\Gamma}\int\limits_{G}\rho^{\,p}(x) d\mu(x)\,.
\end{equation}
Should ${\rm adm\,}\Gamma$ be empty, we set $M_p(\Gamma)=\infty.$ A
family of paths $\Gamma_1$ in $X$ is said to be {\it minorized} by a
family of paths $\Gamma_2$ in $X,$ abbr. $\Gamma_1>\Gamma_2,$ if,
for every path $\gamma_1\in\Gamma_1$, there is a path
$\gamma_2\in\Gamma_1$ such that $\gamma_2$ is a restriction of
$\gamma_1.$ In this case,
\begin{equation}\label{eq32*A}
\Gamma_1
> \Gamma_2 \quad \Rightarrow \quad M_p(\Gamma_1)\le M_p(\Gamma_2)
\end{equation} (см. \cite[Theorem~1]{Fu}).

\medskip
Let $p, q\geqslant 1,$ let $G$ and $G^{\,\prime}$ be domains with
finite Hausdorff dimensions $\alpha$ and $\alpha^{\,\prime}\geqslant
2$ in spaces $(X,d,\mu)$ and $(X^{\,\prime},d^{\,\prime},
\mu^{\,\prime}),$ and let $Q:G\rightarrow[0,\infty]$ be a measurable
function. Given $x_0\in
\partial G,$ denote $S_i:=S(x_0, r_i),$ $i=1,2,$ where $0<r_1<r_2<\infty.$ As in \cite[Ch.~13]{MRSY}, a mapping
$f:G\rightarrow G^{\,\prime}$ (or $f:G\setminus\{x_0\}\rightarrow
G^{\,\prime}$) is a {\it ring $Q$-mapping at a point $x_0\in
\partial G$ with respect to $(p, q)$-moduli}, if the inequality
\begin{equation}\label{eq1C}
M_p(f(\Gamma(S_1, S_2, A)))\leqslant\int\limits_{A\cap
G}Q(x)\eta^q(d(x, x_0))d\mu(x)
\end{equation}
holds for any ring
\begin{equation}\label{eq15A}
A=A(x_0, r_1, r_2)=\{x\in X: r_1<d(x, x_0)<r_2\}, \quad 0 < r_1 <
r_2 <\infty\,,
\end{equation}
and any measurable function
$\eta:(r_1, r_2)\rightarrow [0, \infty]$ such that
\begin{equation}\label{eq8}
\int\limits_{r_1}^{r_2}\eta(r)dr\geqslant 1
\end{equation}
holds. We also consider the definition (\ref{eq1C}) for maps
$f:G\rightarrow X^{\,\prime},$ where $G\subset X$ is a domain of
Hausdorff dimension $\alpha,$ and $X^{\,\prime}$ is a metric space
of Hausdorff dimension $\alpha^{\,\prime}.$ Recall that $X$ is {\it
locally (path) connected} if every neighborhood of a point $x\in X$
contains a (path) connected neighborhood. A space $X$ is called {\it
Ptolemaic}, if for every $x, y, z, t\in X$ we have
\begin{equation}\label{eq1}
d(x, z)d(y, t) + d(x, t)d(y, z) - d(x, y)d(z, t) \geqslant 0\,.
\end{equation}
Following \cite[section~7.22]{He}, given a real-valued function $u$
in a metric space $X,$ a Borel function $\rho\colon X\rightarrow [0,
\infty]$ is said to be an {\it upper gradient} of a function
$u:X\rightarrow{\Bbb R}$ if $|u(x)-u(y)|\leqslant
\int\limits_{\gamma}\rho\,|dx|$ for each rectifiable curve $\gamma$
joining $x$ and $y$ in $X.$ Let $(X, \mu)$ be a metric measure space
and let $1\leqslant p<\infty.$ We say that $X$ admits {\it a $(1;
p)$-Poincar\'{e} inequality} if there is a constant $C\geqslant 1$
and $\tau\geqslant 1$ such that
$$\frac{1}{\mu(B)}\int\limits_{B}|u-u_B|d\mu(x)\leqslant C\cdot({\rm
diam\,}B)\left(\frac{1}{\mu(\tau B)} \int\limits_{\tau B}\rho^p
d\mu(x)\right)^{1/p}$$
for all balls $B$ in $X,$ for all bounded continuous functions $u$
on $B,$ and for all upper gradients $\rho$ of $u.$ Metric measure
spaces where the inequalities $\frac{1}{C}R^{n}\leqslant \mu(B(x_0,
R))\leqslant CR^{n}$
hold for a constant $C\geqslant 1$, every $x_0\in X$  and all
$R<{\rm diam}\,X$, are called {\it Ahlfors $n$-regular.} A metric
space is said to be {\it proper} if its closed balls are compact.

\medskip
Let $G$ be a domain in a space $(X,d,\mu)$. Similarly to \cite{IR},
we say that a function $\varphi:G\rightarrow{\Bbb R}$ has {\it
finite mean oscillation at a point $x_{0}\in\overline{G}$}, abbr.
$\varphi \in FMO(x_{0})$, if
\begin{equation}\label{eq13.4.111} \overline{\lim\limits_{\varepsilon\rightarrow
0}}\,\, \,\frac{1}{\mu(B(x_{0},\varepsilon))}
\int\limits_{B(x_{0},\varepsilon)}|\varphi(x)-\overline{\varphi}_{\varepsilon}|\,\,d\mu(x)<\infty
\end{equation}
where $\overline{\varphi}_{\varepsilon}
=\frac{1}{\mu(B(x_{0},\varepsilon))}
\int\limits_{B(x_{0},\varepsilon)}\varphi(x)\,\,d\mu(x)$ is the mean
value of the function $\varphi(x)$ over the set
$B(x_{0},\varepsilon)=\{x\in G: d(x,x_0)<\varepsilon\}$ with respect
to the measure $\mu$. Here the condition (\ref{eq13.4.111}) includes
the assumption that $\varphi$ is integrable with respect to the
measure $\mu$ over the set $B(x_0,\varepsilon)$ for some
$\varepsilon>0$. Let $X$ and $Y$ be metric spaces. A mapping
$f:X\rightarrow Y$ is discrete if $f^{\,-1}(y)$ is discrete for all
$y\in Y$ and $f$ is open if it takes open sets onto open sets.

Let $D\subset X,$ $f:D \rightarrow X^{\,\prime}$ be a discrete open
mapping, $\beta: [a,\,b)\rightarrow X^{\,\prime}$ be  a curve, and
$x\in\,f^{-1}\left(\beta(a)\right).$ A curve $\alpha:
[a,\,c)\rightarrow D$ is called a {\it maximal $f$-lifting} of
$\beta$ starting at $x,$ if $(1)\quad \alpha(a)=x\,;$ $(2)\quad
f\circ\alpha=\beta|_{[a,\,c)};$ $(3)$\quad for
$c<c^{\prime}\leqslant b,$ there is no curves $\alpha^{\prime}:
[a,\,c^{\prime})\rightarrow D$ such that
$\alpha=\alpha^{\prime}|_{[a,\,c)}$ and $f\circ
\alpha^{\,\prime}=\beta|_{[a,\,c^{\prime})}.$ In the case
$X=X^{\,\prime}={\Bbb R}^n,$ the assumption on $f$ yields that every
curve $\beta$ with $x\in f^{\,-1}\left(\beta(a)\right)$ has  a
maximal $f$-lif\-ting starting at $x$ (see
\cite[Corollary~II.3.3]{Ri}). Consider the condition

\medskip
$\textbf{A}:$ {\bf for all $\beta: [a,\,b)\rightarrow X^{\,\prime}$
and $x\in f^{\,-1}\left(\beta(a)\right),$ a mapping $f:D\rightarrow
X^{\,\prime}$ has a maximal $f$-lif\-ting in $D$ starting at $x.$}
The main result is the following theorem.

\medskip
\begin{theorem}\label{th2}{\sl\,
Let $2\leqslant\alpha, \alpha^{\,\prime}<\infty,$ let
$\alpha^{\,\prime}-1<p\leqslant\alpha$ and $1\leqslant q\leqslant
\alpha,$ let $(X, d, \mu)$ be locally compact metric space, and let
$X^{\,\prime}$ be an Ahlfors $\alpha^{\,\prime}$-regular, proper,
path connected, locally connected and Ptolemaic metric space in
which the $(1; p)$-Poincar\'{e} inequality is fulfilled. Let
$G:=D\setminus\{\zeta_0\}$ be a domain in $X$ of Hausdorff dimension
$\alpha,$ which is locally path connected at $\zeta_0\in D.$ Assume
that $Q\in FMO(\zeta_0).$

If an open discrete ring $Q$-mapping $f\colon
D\setminus\{\zeta_0\}\rightarrow X$ at $\zeta_0$ with respect to
$(p, q)$-moduli satisfies the condition $\textbf{A}$ and $\zeta_0$
is an essential singularity of $f,$ then the following condition
holds: for every $A\in X^{\,\prime}$ there exists $x_k\in
D\setminus\{\zeta_0\},$ $x_k\rightarrow \zeta_0$ as
$k\rightarrow\infty,$ such that $d^{\,\prime}(f(x_k), A)\rightarrow
0$ as $k\rightarrow\infty.$ }
\end{theorem}

\section{An analog of spherical metric in metric spaces}

Now we give an analog of known spherical (chordal) metric in metric
spaces. This analog was firstly introduced in \cite{Kl} for linear
normalized spaces. Given a point $x_0\in X,$ set
\begin{equation}\label{eq2.2.1}
h_{x_0}(x,y):=\frac{d(x, y)}{\sqrt{1+d^2(x, x_0)}\sqrt{1+d^2(y,
x_0)}}\,.
\end{equation}

The following statement was proved in \cite{Kl} in the case of
linear normalized spaces.

\medskip
\begin{lemma}\label{lem1}
{\sl Let $(X, d)$ be a Ptolemaic metric space. If $\alpha>0,$
$\beta\geqslant 0$ and $p\geqslant 1,$ then
\begin{equation}\label{eq18}
H_{x_0}(x,y):=\frac{d(x, y)}{(\alpha+\beta d^p(x,
x_0))^{1/p}(\alpha+\beta d^p(y, x_0))^{1/p}}
\end{equation}
is a metric on $X.$ In particular, $h_{x_0}(x, y)$ can be obtained
from (\ref{eq18}) by the setting $\alpha=\beta=1$ and $p=2;$ thus,
$h_{x_0}(x, y)$ is a metric on $X.$}
\end{lemma}

\medskip
\begin{proof} We need to prove the triangle inequality, only. Put
$x,y,z\in X.$ We need to prove that
\begin{equation}\label{eq2}
H_{x_0}(x, z)\leqslant H_{x_0}(x, y)+H_{x_0}(y, z)\,.
\end{equation}
Since $d$ is a metric on $X,$
\begin{equation}\label{eq3}
\alpha (d(x, y)+d(y, z))^p\geqslant \alpha d^p(x, z)\,.
\end{equation}
From other hand, by Minkowski's inequality
\begin{equation}\label{eq1B}\left(\sum\limits_{k=1}^n |x_k+y_k|^p\right)^{1/p}\leqslant
\left(\sum\limits_{k=1}^n |x_k|^p\right)^{1/p}+
\left(\sum\limits_{k=1}^n |y_k|^p\right)^{1/p}\,.
\end{equation}
Now, we put $n=2,$ and
$$X=(x_1, x_2)=(\alpha^{1/p}\cdot d(x,y),\, \beta^{1/p}\cdot d(x,y)\cdot d(x_0,
z))\in {\Bbb R}^2,$$
$$Y=(y_1, y_2)=(\alpha^{1/p}\cdot d(y,z),\, \beta^{1/p}\cdot
d(y,z)\cdot d(x_0, x))\in {\Bbb R}^2\,.$$
By (\ref{eq1}), (\ref{eq3}) and (\ref{eq1B}), we have
$$d(x, y)(\alpha+\beta d^p(x_0, z))^{1/p}+d(y, z)(\alpha+\beta d^p(x_0, x))^{1/p}\geqslant$$
\begin{equation}\label{eq4}
\geqslant\left(\alpha(d(x,
y)+d(y,z))^p+\beta(d(x,y)d(z,x_0)+d(y,z)d(x_0,
x))^p\right)^{1/p}\geqslant
\end{equation}
$$\geqslant\left(\alpha d^p(x, z)+ \beta d^p(x,z)d^p(y,x_0)\right)^{1/p}=d(x, z)(\alpha+\beta d^p(y, x_0))^{1/p}\,.$$
Dividing (\ref{eq4}) on $(\alpha+\beta d^p(x_0,
z))^{1/p}\cdot(\alpha+\beta d^p(y, x_0))^{1/p}\cdot(\alpha+\beta
d^p(x_0, x))^{1/p},$ we obtain that
$$\frac{d(x, y)}{(\alpha+\beta d^p(y, x_0))^{1/p}\cdot(\alpha +\beta d^p(x_0, x))^{1/p}}+
\frac{d(y, z)}{(\alpha+\beta d^p(y, x_0))^{1/p}\cdot(\alpha+\beta
d^p(x_0, z))^{1/p}}\geqslant$$$$\geqslant \frac{d(x,
z)}{(\alpha+\beta d^p(z, x_0))^{1/p}\cdot(\alpha+\beta d^p(x_0,
x))^{1/p}}\,,$$
or, in other words, (\ref{eq2}), as required.~$\Box$
\end{proof}

\medskip
\begin{remark}\label{rem1}
It is easy to see that, metrics $H_{x_0}(x,y)$ are equivalent in $X$
under different $\alpha, \beta$ and $p.$ Thus, we restrict us by
studying of the metric (\ref{eq2.2.1}), only.
\end{remark}

\medskip The {\it spherical (chordal)}
diameter of a set $E\subset X$ is
$$
h_{x_0}(E)\,=\,\sup\limits_{x\,,y\,\in\,E}\,h_{x_0}(x,y)\,.
$$
Now we have $h_{x_0}(X)\leqslant 1.$ The following nearly obvious
lemma can be useful for our further studying.

\medskip
\begin{lemma}\label{lem3}
{\sl Let $(X, d)$ be a Ptolemaic metric space, and let $C$ be a
compact in $(X, d).$ Now, $C$ is a compact in $(X, h_{x_0}),$
moreover, there exist $\zeta_0, y_0\in C$ with
\begin{equation}\label{eq15}
h_{x_0}(C)=h_{x_0}(\zeta_0, y_0)\,.
\end{equation} }
\end{lemma}
\begin{proof}
Let $C$ be a compact in $(X, d),$ and let $x_k\in C.$ By the
definition, we can find $x_{k_l}$ and $z_0\in X$ such that
$d(x_{k_l}, z_0)\rightarrow 0$ as $l\rightarrow\infty.$ Since
$h_{x_0}(x, y)\leqslant d(x, y)$ for every $x, y\in X,$ we obtain
that $h_{x_0}(x_{k_l}, z_0)\rightarrow z_0$ as $l\rightarrow\infty.$
Thus, $C$ is a compact in $(X, h_{x_0}).$

Let us to prove (\ref{eq15}). By the definition of $\sup,$ for every
$k=1,2,\ldots$ there exist $x_k, y_k\in C$ with
\begin{equation}\label{eq16}
h_{x_0}(C)-1/k\leqslant h_{x_0}(x_k, y_k)\leqslant h_{x_0}(C)\,.
\end{equation}
Thus, $h_{x_0}(x_k, y_k)\rightarrow 0$ as $k\rightarrow\infty.$
Since $C$ is a compact in $(X, h_{x_0}),$ we can assume that
$h_{x_0}(x_k,\zeta_0)\rightarrow 0$ as $k\rightarrow\infty$ and
$h_{x_0}(y_k, y_0)\rightarrow 0$ as $k\rightarrow\infty$ for some
$\zeta_0, y_0\in C.$ By triangle inequality, $h_{x_0}(x_k,
y_k)-h_{x_0}(\zeta_0, y_0)\leqslant h_{x_0}(x_k,
\zeta_0)+h_{x_0}(y_k, y_0)$ and, simultaneously, $h_{x_0}(\zeta_0,
y_0)-h_{x_0}(x_k, y_k)\leqslant h_{x_0}(x_k, \zeta_0)+h_{x_0}(y_k,
y_0).$ Thus, we obtain that
\begin{equation}\label{eq17}
|h_{x_0}(x_k, y_k)-h_{x_0}(\zeta_0, y_0)|\leqslant h_{x_0}(x_k,
\zeta_0)+h_{x_0}(y_k, y_0)\rightarrow 0,\quad k\rightarrow\infty\,.
\end{equation}
By (\ref{eq16}) and (\ref{eq17}), we obtain (\ref{eq15}), as
required.~$\Box$
\end{proof}

\medskip
\section{On capacity estimates through chordal diameter}

Classic capacity estimates were proved for conformal capacity in
${\Bbb R}^n$ in \cite[Lemma~3.11]{MRV$_2$} or
\cite[Lemma~2.6.III]{Ri}. Also, we have obtained some analogs of
capacity estimates of order $p$ in \cite[Lemma~2.1]{GSS}. Our main
goal now is to extend the results mentioned above for metric spaces.

\medskip As usually, given a curve $\gamma:[a, b]\rightarrow X,$ we set
$$|\gamma|:=\{x\in X: \exists\,t\in[a, b]:
\gamma(t)=x\}\,.$$
Recall that a pair $E=\left(A,\,C\right),$ where $A$ is an open set
in $X,$ and $C$ is a compact subset of $A,$ is called {\it
condenser} in $X.$ Given $p\geqslant 1,$ a quantity
\begin{equation}\label{eq28}
{\rm cap}_p\,E=M_p(\Gamma_E)
\end{equation}
is called {\it $p$-capacity of $E,$} where $\Gamma_E$ be the family of all paths of the form $\gamma\colon
[a,\,b)\rightarrow A$ with $\gamma(a)\in C$ and
$|\gamma|\cap\left(A\setminus F\right)\ne\varnothing$ for every
compact $F\subset A.$

\medskip
The following result holds (see \cite[Proposition~4.7]{AS}).

\medskip
\begin{proposition}\label{pr2}
{\sl Let $X$ be a $Q$-Ahlfors regular metric measure space that
supports $(1; p)$-Poincar\'{e} inequality for some $p>1$ such that
$Q-1<p\leqslant Q.$ Let $E$ and $F$ be continua contained in a ball
$B_R\subset X.$ Then
$$M_p(\Gamma(E, F, X))\geqslant \frac{1}{C}\cdot\frac{\min\{{\rm diam}\,E, {\rm diam}\,F\}}{R^{1+p-Q}}$$
for some constant $C>0.$
 }
\end{proposition}

\medskip
Let us to prove the following statement.

\medskip
\begin{lemma}\label{lem6}
{\sl\, Let $X$ be a Ptolemaic metric space, let $a>0,$ and let $F$
be a non-degenerate continuum in $X.$ Assume that, $C$ is some
continuum in $X\setminus F$ with $h_{x_0}(C)\geqslant a,$ and $R>0$
is some number with $h_{x_0}\left(X\setminus B(x_0, R)\right)< a/2.$
Now, there exists continuum $C_1\subset C\cap \overline{B(x_0, R)}$
such that $h_{x_0}(C_1)\geqslant a/4.$}
\end{lemma}

\begin{proof} Be Lemma \ref{lem1}, $h_{x_0}$ is a metric.

If $C\subset\overline{B(x_0, R)},$ then we put $C_1:=C.$ Now, assume
that there exists $z_0\in C\cap \left(X\setminus \overline{B(x_0,
R)}\right).$ Since $C$ is a compact, by Lemma \ref{lem3} there exist
$\zeta_0, y_0\in C$ such that $h_{x_0}(C)=h_{x_0}(\zeta_0, y_0).$
Observe that $\zeta_0$ and $y_0$ do not both belong to the
complement of $B(x_0, R),$ since $h_{x_0}(C)\geqslant a,$ while
$h_{x_0}\left(X\setminus B(x_0, R)\right)< a/2.$ Let $\zeta_0\in
B(x_0, R).$ There are two possibilities:

1) $y_0\in X\setminus B(x_0, R).$ Let $C_2$ be $\zeta_0$-component of $C\cap
\overline{B(x_0, R)}.$ Since $C$ is
connected and $C\setminus B(x_0, r)\ne \varnothing,$ $C_2\cap
{\overline{C\setminus C_2}}\ne \varnothing$ (see \cite[item~1,
$\S\,46,$ Ch.~5]{Ku}). Observe that
\begin{equation}\label{eq2A}C\setminus C_2=(C\setminus \overline{B(x_0,
R)})\cup\bigcup\limits_{\alpha\in A} K_{\alpha}\,,
\end{equation}
where $A$ is some set of indexes $\alpha,$ and
$\bigcup\limits_{\alpha\in A} K_{\alpha}$ is the union of all
components of $C\cap \overline{B(x_0, R)},$ excluding $C_2.$ By
\cite[Theorem 1.III, $\S\,46,$ Ch.~5]{Ku}), $K_{\alpha}$ and $C_2$
are closed disjoint sets in $\overline{B(x_0, R)},$ $\alpha\in A.$
Thus, by (\ref{eq2A}), $C_2\cap {\overline{C\setminus C_2}}\ne
\varnothing$ is possible if and only if $C_2\cap
\overline{(C\setminus \overline{B(x_0, R)})}\ne\varnothing.$ Now,
there exists $z_1\in C_2\cap S(x_0, R).$ By triangle inequality
$$a\leqslant h_{x_0}(\zeta_0, y_0)\leqslant h_{x_0}(\zeta_0, z_1)+h_{x_0}(z_1, y_0)< h_{x_0}(C_2)+a/2\,,$$
whence we obtain that $h_{x_0}(C_2)> a/2,$ as required. Let us consider the
second case: assume that

\medskip
2) $y_0\in B(x_0, R).$ Let $C_2$ be $\zeta_0$-component of $C\cap
\overline{B(x_0, R)}.$ Denote $C_3$ the $y_0$-component of $C\cap \overline{B(x_0, R)}.$
Arguing is in the case 1, we obtain that there exists $z_2\in C_3\cap S(x_0, R).$ By triangle inequality
$$a\leqslant h_{x_0}(\zeta_0, y_0)\leqslant h_{x_0}(\zeta_0, z_1)+h_{x_0}(z_1, z_2)+
h_{x_0}(z_2, y_0)< h_{x_0}(C_2) + h_{x_0}(C_3)+ a/2\,,$$
whence we obtain that either $h_{x_0}(C_2)>a/4,$ or $h_{x_0}(C_3)>
a/4,$ as required.~$\Box$
\end{proof}

\medskip
An analog of the following lemma was proved in ${\Bbb R}^n$ in
\cite[Lemma~3.11]{MRV$_2$}, see also \cite[Lemma~2.6.III]{Ri} and
\cite[Lemma~2.1]{GSS}).

\medskip
\begin{lemma}\label{lem2}
{\sl\, Let $\alpha\geqslant 2,$ let $\alpha-1<p<\alpha,$ and let $X$
be $\alpha$-Ahlfors regular, path connected, locally connected,
locally compact and Ptolemaic metric measure space that supports
$(1; p)$-Poincar\'{e} inequality. Assume that $F$ is nondegenerate
continuum in $X.$ Now, for every $a>0$ there exists $\delta>0$ the
following condition holds:
\begin{equation}\label{eq5}{\rm cap}_p\,\left(X\setminus F,\, C\right)\geqslant \delta\end{equation}
for every continuum $C\subset
X\setminus F$ with $h_{x_0}(C)\geqslant a.$
 }
\end{lemma}

\begin{proof}
By Lemma \ref{lem1}, $h_{x_0}$ is a metric on $X.$
There are two possibilities:

\medskip
{\bf 1) Assume that $X$ is bounded,} i.e., there exists $R_0>0$ such that
$X=B(x_0, R_0).$ Let $\Gamma_0$ be a family of all curves, for which $\alpha$-capacity in (\ref{eq5})
is attained. In other words, let $\Gamma_0$ be the family of all curves $\gamma\colon [a,\,b)\rightarrow X\setminus F,$
such that $\gamma(a)\in C$ and $|\gamma|\cap\left((X\setminus
F)\setminus F_0\right)\ne\varnothing$ for every compact
$F_0\subset X\setminus F.$
\medskip
We show that
\begin{equation}\label{eq21}
\Gamma(C, F, X)>\Gamma(C, F, X\setminus F)\,.
\end{equation}
Indeed, let $\alpha\in \Gamma(C, F, X),$ $\alpha:[a, b]\rightarrow X,$
$\alpha(a)\in C,$ $\alpha(b)\in F$ and $\alpha(t)\in X$ for each $t\in (a, b).$ Put
$$c:=\inf\{t\in[a, b]: \alpha(t)\in F\}\,.$$
Observe that $a<c\leqslant b.$ In fact, suppose the contrary, i.e., assume that $c=a.$ Now, there exists
$t_k\rightarrow a+0$ as
$k\rightarrow\infty$ with $\alpha(t_k)\in F.$ Now $\alpha(t_k)\rightarrow
\alpha(a)\in C$ as $k\rightarrow\infty$ by continuity of $\alpha.$ Thus,
$\alpha(a)\in C\cap F,$ because $C$ and $F$ are continua in $X.$ This contradicts with definition of $E$ and $F.$
Thus, $c>a,$ as required.

\medskip
Put $\alpha|_{[a, c]}.$ Observe that $\alpha\in\Gamma(C, F,
X\setminus F),$ thus, (\ref{eq21}) holds, as required.

\medskip
Let $\Gamma_1$ be a family of all half-open curves $\alpha|_{[a,
c)},$ where $\alpha\in \Gamma(C, F, X\setminus F).$ Observe that
$\Gamma(C, F, X\setminus F)>\Gamma_1.$ We show that
\begin{equation}\label{eq22}
\Gamma_1\subset\Gamma_0\,.
\end{equation}
Assume the contrary, i.e., assume that (\ref{eq22}) does not hold. Now, there exists
$\gamma_1\in\Gamma_1$ and a compact
$F_1\subset X\setminus F$ such that $|\gamma_1|\cap((X\setminus
F)\setminus F_1)=\varnothing.$ Now, we obtain that
$|\gamma_1|\subset F_1.$ Since $|\gamma_1|$ and $F$ are disjoint compacts in $X,$
${\rm dist\,}(|\gamma_1|,
F)>0.$ This contradicts with the condition
$\gamma(t)\rightarrow\gamma(c)$ as $t\rightarrow c-0.$ Thus, (\ref{eq22}) holds, as required.

\medskip
We obtain from (\ref{eq21}) and (\ref{eq22}) that
$$\Gamma(C, F, X)>\Gamma(C, F, X\setminus F)>\Gamma_1\subset\Gamma_0\,.$$
Now, by properties of $p$-modulus
\begin{equation}\label{eq23}
M_p(\Gamma(C, F, X))\leqslant {\rm cap}_p\,\left(X\setminus F,\,
C\right)\,.
\end{equation}
From other hand, by Proposition \ref{pr2}
\begin{equation}\label{eq24}
M_p(\Gamma(C, F, X))\geqslant \frac{1}{C}\cdot\frac{\min\{{\rm
diam}\,C, {\rm diam}\,F\}}{R^{1+p-\alpha}}\geqslant C_1\cdot a\,,
\end{equation}
where $C_1$ depends only on $F,$ $R,$ $\alpha$ and $p.$ Put
$\delta:=C_1\cdot a.$
Comparing (\ref{eq23}) and (\ref{eq24}), we obtain
(\ref{eq5}), as required.

\medskip
Consider the most difficult second situation:

\medskip
{\bf 2) $X$ is unbounded}, i.e., given $R>0,$ there exists $x\in X$ such that $x\in X\setminus
B(x_0, R).$ Since $F$ is a compact in $X,$ there exists $R>0$ with $F\subset B(x_0, R).$ Observe that
\begin{equation}\label{eq26}
h_{x_0}(x, y)\leqslant \frac{1}{\sqrt{1+d^2(x_0,
y)}}+\frac{1}{\sqrt{1+d^2(x_0, x)}}\,.
\end{equation}
Thus, $h_{x_0}(X\setminus B(x_0, R))\rightarrow 0$ as
$R\rightarrow\infty.$ So, we can find sufficiently large $R,$ such that
\begin{equation}\label{eq6}
h_{x_0}\left(X\setminus B(x_0, R)\right)< a/2\,.
\end{equation}
By Lemma \ref{lem6}, there is a subcontinuum $C_1\subset C$ with
$C_1\subset \overline{B(x_0, R)},$ such that $h_{x_0}(C_1)\geqslant
a/4.$ Observe that, by the definition of $p$-capacity in
(\ref{eq4}), ${\rm cap}_p\,\left(X\setminus F,\, C\right)\geqslant
{\rm cap}_p\,\left(X\setminus F,\, C_1\right).$ Thus, it is
sufficiently to find the lower estimate for ${\rm
cap}_p\,\left(X\setminus F,\, C_1\right).$

\medskip
Since $X$ is unbounded, there exists $z_0\in X\setminus \overline{B(x_0, 2R)}.$
Let $t_0>0$ be such that $B(z_0, t_0)\subset X\setminus \overline{B(x_0,
2R)}.$ Since $X$ is a locally connected and locally compact space, we can consider that
$\overline{B(z_0, t_0)}$ is a compact in $X.$ Put $t_*<t_0.$ Since $X$ is locally connected, there exists
a connected neighborhood $V_0$ of $z_0.$ In particular, there exists
$t_1>0,$ $t_1<t_*,$ such that $B(z_0, t_1)\subset V_0.$
Thus, $B(z_0, t_1)\subset V_0\subset B(z_0, t_*),$ and, consequently,
$\overline{B(z_0, t_1)}\subset \overline{V_0}\subset
\overline{B(z_0, t_*)}.$ Now, we obtain that
\begin{equation}\label{eq20}
B(z_0, t_1)\subset V\subset B(z_0, t_0)\,,
\end{equation}
where $V=\overline{V_0}$ is the continuum in $X.$ Observe that $B(z_0,
t_1)$ is not degenerate into a point, because
$X$ is Ahlfors
regular. Thus, (\ref{eq20}) implies that the continuum $V$ is non-degenerate.

\medskip
Let $B=B(R)$ such that $B>R$ and $\overline{B(z_0, t_0)}\subset
B(x_0, B).$ (For instance, we can put $B_2:=d(x_0, z_0)+t_0$).
Denoting $\Gamma_1=\Gamma(F, V, B(x_0, B)),$
$\Gamma_2=\Gamma\left(C_1, V, B(x_0, B)\right),$  by Proposition
\ref{pr2} we obtain that
\begin{equation}\label{eq19}
M_p(\Gamma_1)\geqslant \frac{1}{C}\cdot\frac{\min\{{\rm diam}\,F,
{\rm diam}\,V\}}{B^{1+p-\alpha}}\geqslant\delta_1
\end{equation}
and
\begin{equation}\label{eq19A}
M_p(\Gamma_2)\geqslant \frac{1}{C}\cdot\frac{\min\{{\rm diam}\,C_1,
{\rm diam}\,V\}}{B^{1+p-\alpha}}\geqslant\delta_1
\end{equation}
where $\delta_1$ depends only on $F,$ $R,$ $\alpha,$ $p$ and $V,$
and $\delta_2$ depends only on $a,$ $R,$ $\alpha,$ $p$ and $V.$

\medskip
Denote
$$\Gamma_{1,2}=\Gamma\left(C_1, F, X\right)\,.$$
Arguing as in the proof of (\ref{eq23}),
we observe that
\begin{equation}\label{eq9}
M_p(\Gamma_{1,2})\leqslant {\rm cap}_p\,\left(X\setminus F,\,
C_1\right)\,.
\end{equation}
Let $\rho\in {\rm adm\,}\Gamma_{1,2}.$ If $3\rho\in {\rm
adm\,}\Gamma_{1},$ or, if $3\rho\in  {\rm adm\,}\Gamma_{2},$ then we
obtain from (\ref{eq19}) and (\ref{eq19A}) that
\begin{equation}\label{eq10}
\int\limits_{X}\rho^p(x)d\mu(x)\geqslant 3^{\,-p}\min\{\delta_1,
\delta_2\}\,.
\end{equation}
Assume that $3\rho\not\in  {\rm adm\,}\Gamma_{1}$ and, simultaneously,
$3\rho\not\in {\rm adm\,}\Gamma_{2}.$ Now, there exist
$\gamma_1\in \Gamma_1$ and $\gamma_2\in \Gamma_2$ such that
\begin{equation}\label{eq11}
\int\limits_{\gamma_1}\rho (x)\ \ |dx|< 1/3,\qquad
\int\limits_{\gamma_2}\rho (x)\ \ |dx|< 1/3\,.
\end{equation}
Recall that $F, C_1\subset B(x_0, 2R)$ and $V\subset X\setminus
B(x_0, 2R).$ Now, by \cite[Theorem 1, $\S\,$46, item I]{Ku} there
exist $\widetilde{\gamma_1},$ $\widetilde{\gamma_2}\in \Gamma(S(x_0,
R), S(x_0, 2R), B(x_0, 2R))$ such that $\widetilde{\gamma_i}$ are
subcurves of $\gamma_i,$ $i=1,2.$ Observe that ${\rm
diam\,}\gamma_i\geqslant R.$ Putting
$$\Gamma_4=\Gamma\left(|\gamma_1|, |\gamma_2|,
X\right)\,,$$ we obtain that
\begin{equation}\label{eq21A}
\Gamma\left(|\widetilde{\gamma_1}|, |\widetilde{\gamma_2}|,
X\right)\subset \Gamma_4\,.
\end{equation}
Moreover, by Proposition \ref{pr2}
\begin{equation}\label{eq20A}
M_p(\Gamma\left(|\widetilde{\gamma_1}|, |\widetilde{\gamma_2}|,
X\right))\geqslant\frac{1}{C}\cdot\frac{\min\{{\rm
diam}\,|\widetilde{\gamma_1}|, {\rm
diam}\,|\widetilde{\gamma_2}|\}}{R^{1+p-\alpha}}\geqslant\frac{2R^{\alpha-p}}{C}\,.
\end{equation}
We obtain from (\ref{eq21A}) and (\ref{eq20A}) that
\begin{equation}\label{eq22A}
M_p(\Gamma_4)\geqslant 2/C\,.
\end{equation}
From other hand, we obtain from (\ref{eq11}) that $3\rho\in {\rm
adm\,}\Gamma_{4}.$ Now by (\ref{eq20}) we obtain that
\begin{equation}\label{eq13}
\int\limits_{X}\rho^{\,p}(x)d\mu(x)\geqslant
2R^{\alpha-p}\cdot3^{-p}/C\,.
\end{equation}
Finally, by (\ref{eq10}) and (\ref{eq13}), we obtain
\begin{equation}\label{eq14}
M_p(\Gamma_{1,2})\geqslant 3^{\,-p}\min\{\delta_1, \delta_2,
2R^{\alpha-p}/C\}:=\delta\,.
\end{equation}
Thus, (\ref{eq5}) follows from (\ref{eq14}) and
(\ref{eq9}), as required.~$\Box$
\end{proof}

\medskip
\section{The main lemma}

\medskip
The following lemma have been proved in \cite[Lemma~5]{Sev$_2$} for
$p=\alpha$ and $q=\alpha^{\,\prime}.$

\medskip
\begin{lemma}\label{lem3.1!}
{\sl\, Let $2\leqslant\alpha, \alpha^{\,\prime}<\infty,$ let $p,
q\geqslant 1,$ let $D$ be a domain in $(X, d, \mu)$ of Hausdorff
dimension $\alpha\geqslant 2,$ and let $(X^{\,\prime}, d^{\,\prime},
\mu^{\,\prime})$ be a metric space of Hausdorff dimension
$\alpha^{\,\prime}\geqslant 2.$ Suppose that there exists
$\varepsilon_0>0$ and a Lebesgue measurable function
$\psi(t)\colon(0, \varepsilon_0)\rightarrow [0,\infty]$ with the
following property: for every $\varepsilon_2\in(0, \varepsilon_0]$
there is $\varepsilon_1\in (0, \varepsilon_2]$ such that
\begin{equation} \label{eq5C}
0<I(\varepsilon,
\varepsilon_2):=\int\limits_{\varepsilon}^{\varepsilon _2}\psi(t)dt
< \infty
\end{equation}
for every $\varepsilon\in (0,\varepsilon_1).$ Assume also that
\begin{equation} \label{eq4*}
\int\limits_{\varepsilon<d(x,
x_0)<\varepsilon_0}Q(x)\cdot\psi^q(d(x, x_0)) \
d\mu(x)\,=\,o\left(I^q(\varepsilon, \varepsilon_0)\right)\,.
\end{equation}
as $\varepsilon\rightarrow 0.$

Let $\Gamma$ be the family of curves
$\gamma(t)\colon(0,1)\rightarrow D\setminus\{x_0\}$ such that
$\gamma(t_k)\rightarrow x_0$ as $k\rightarrow\infty$ for some
sequence $t_k\rightarrow 0,$ $\gamma(t)\not\equiv x_0,$ and let
$f\colon D\setminus\{x_0\}\rightarrow X^{\,\prime}$ be a ring
$Q$-mapping at $x_0\in D$ with respect to $(p, q)$-moduli. Then
$M_p\left(f(\Gamma)\right)=0.$ }
 \end{lemma}

\medskip
In particular, \eqref{eq5C} holds true whenever a given function
$\psi\in L^1_{loc}(0, \varepsilon_0)$ satisfies the condition
$\psi(t)>0$ for almost every $t\in (0, \varepsilon_0).$

\begin{proof}
We observe that
\begin{equation}\label{eq12*}
\Gamma > \bigcup\limits_{i=1}^\infty\,\, \Gamma_i\,,
\end{equation}
where is the family of all curves $\alpha_i(t)\colon(0,1)\rightarrow
D\setminus\{x_0\}$ such that $\alpha_i(1)\in \{0<d(x,
x_0)=r_i<\varepsilon_0\},$  where $r_i$ is a sequence with
$r_i\rightarrow 0$ as $i\rightarrow \infty$ and
$\alpha_i(t_k)\rightarrow x_0$ as $k\rightarrow\infty$ for the above
sequence $t_k,$ $t_k\rightarrow 0$ as $k\rightarrow\infty.$ Fix
$i\geqslant 1.$ By (\ref{eq5C}), we see that $I(\varepsilon, r_i)>0$
for all $\varepsilon\in(0, \varepsilon_1)$ with some
$\varepsilon_1\in (0, r_i].$ Now, we observe that the function
$$\eta(t)=\left\{
\begin{array}{rr}
\psi(t)/I(\varepsilon, r_i), &   t\in (\varepsilon,
r_i),\\
0,  &  t\in {\Bbb R}\setminus (\varepsilon, r_i)
\end{array}
\right. $$
satisfies (\ref{eq8}) in the ring $A(x_0, \varepsilon, r_i)=\{x\in
X: \varepsilon<d(x, x_0)< r_i \}.$ Since $f$ is a ring $Q$-mapping
at $x_0$ with respect to $(p, q)$-moduli we obtain
\begin{multline}\label{eq11*}
M_p\left(f\left(\Gamma\left(S(x_0, \varepsilon),\,S(x_0,
r_i),\,A(x_0, \varepsilon, r_i)\right)\right)\right)\leqslant\\
\leqslant \int\limits_{A(x_0, \varepsilon, r_i)} Q(x)\cdot
\eta^q(d(x, x_0))\ d\mu(x)\,\leqslant {\frak F}_i(\varepsilon),
 \end{multline}
where
${\frak F}_i(\varepsilon)=\,\frac{1}{\left(I(\varepsilon,
r_i)\right)^q}\int\limits_{\varepsilon<d(x,
x_0)<\varepsilon_0}\,Q(x)\,\psi^q(d(x, x_0))\,d\mu(x).$ By
(\ref{eq4*}), ${\frak F}_i(\varepsilon)\rightarrow 0$ as
$\varepsilon\rightarrow 0.$ Note that
\begin{equation}\label{eq5*C}
\Gamma_i>\Gamma\left(S(x_0, \varepsilon),\,S(x_0, r_i),\,A(x_0,
\varepsilon, r_i)\right)
\end{equation}
for every $\varepsilon\in (0, \varepsilon_1).$ By (\ref{eq11*}) and
(\ref{eq5*C}), we have
\begin{equation}\label{eq6*}
M_p(f(\Gamma_i))\leqslant {\frak F}_i(\varepsilon)\rightarrow 0
\end{equation}
for every fixed $i=1,2,\ldots,$ and $\varepsilon\rightarrow 0.$ But
the left-hand side of (\ref{eq6*}) does not depend on $\varepsilon,$
whence we see that $M_p(f(\Gamma_i))=0.$ Finally, by (\ref{eq12*})
and the semiadditivity of the modulus of a family of curves (see
\cite[10, Theorem~1(b)]{Fu}), we obtain $M_p(f(\Gamma))=0,$ as
required.
\end{proof}

\medskip
Set
$\overline{X}:=X\cup\infty$ and
$$h_{x_0}(x, \infty)=\frac{1}{\sqrt{1+d^2(x_0, x)}}\,.$$
It is not difficult to see that $h_{x_0}$ is a metric on $\overline{X}.$
Indeed, by Lemma \ref{lem1}, $h_{x_0}$ is a metric on $X.$ By (\ref{eq26}), we obtain $h_{x_0}(x, y)\leqslant
h_{x_0}(x, \infty)+h_{x_0}(\infty, y)$ for every $x, y\in X.$ We show that
\begin{equation}\label{eq27}h_{x_0}(x, \infty)\leqslant h_{x_0}(x, y)+h_{x_0}(y, \infty)
\end{equation}
for every $x, y\in X.$ Using the definition of $h_{x_0}(x, y),$ we
obtain that (\ref{eq27}) is equivalent to $\sqrt{1+d^2(x_0,
y)}\leqslant d(x, y)+\sqrt{1+d^2(x_0, x)}.$ Since by triangle
inequality $d(x_0, y)\leqslant d(x_0, x)+d(x, y),$ we need to prove
that $\sqrt{1+(d(x_0, x)+d(x, y))^2}\leqslant d(x,
y)+\sqrt{1+d^2(x_0, x)}.$ Denoting $a=d(x_0, x)$ and $b=d(x, y),$ we
rewrite this relation in the form
$\sqrt{1+(a+b)^2}\leqslant b+\sqrt{1+a^2},$ or, equivalently, $2ab\leqslant 2b\sqrt{1+a^2}.$
Since the last relation is obvious, (\ref{eq27}) holds, as required. Another properties of a metric
for $h_{x_0}$ are obvious.

\medskip
The following statement holds.

\medskip
\begin{lemma}\label{lem5}
{\sl If $(X, d)$ is proper and Ptolemaic, then $(\overline{X},
h_{x_0})$ is compact.}
\end{lemma}

\medskip
\begin{proof} By Lemma \ref{lem1} and remarks mentioned above, $h_{x_0}$ is a metric on $\overline{X}.$
Put $x_n\in \overline{X},$ $n=1,2,\ldots, .$ We need to prove that there exists
$x_{n_k}$ such that $x_{n_k}\rightarrow x_0$ for some $x_0\in \overline{X}$ as
$k\rightarrow\infty.$ If $x_n=\infty$ for infinitely large $n,$
the statement of Lemma holds for $x_0=\infty.$

Now, assume that $x_n\ne \infty$ for each $n\geqslant N_0$ and some
$N_0\in {\Bbb N}.$ There are two possibilities: 1) for every $m>0$
there is $x_{n_m}\in X\setminus B(x_0, m).$ By the definition of
$h_{x_0},$ we obtain that $h_{x_0}(x_{n_m}, \infty)\rightarrow 0$ as
$m\rightarrow\infty.$ 2) There exists $R>0$ for which $x_n\in
\overline{B(x_0, R)},$ $n=1,2,\ldots ,.$ Since $(X, d)$ is
proper, $\overline{B(x_0, R)}$ is a compact in $(X, d).$ Thus, there
exist $x_{l_k}$ and $z_0\in
\overline{B(x_0, R)}$ such that $d(x_{l_k}, z_0)\rightarrow 0,$
$k\rightarrow\infty.$ Since $h_{x_0}(x, y)\leqslant d(x, y)$
for every $x,y\in X,$ we obtain that $h_{x_0}(x_{l_k}, z_0)\rightarrow 0$ as
$k\rightarrow\infty.$ Lemma is proved.~$\Box$
\end{proof}

\medskip
The next lemma is a statement about removable singularities of open discrete mappings
in the most general setting.

\medskip
\begin{lemma}\label{lem4*}{\sl\, Let $2\leqslant\alpha, \alpha^{\,\prime}<\infty,$ let
$\alpha^{\,\prime}-1<p\leqslant\alpha$ and $1\leqslant q\leqslant
\alpha.$
Let $G:=D\setminus\{\zeta_0\}$ be a domain in a locally compact
metric space $(X, d, \mu)$ of Hausdorff dimension $\alpha,$ where
$G$ is locally path connected at $\zeta_0\in D,$ and let
$(X^{\,\prime}, d^{\,\prime}, \mu^{\,\prime})$ be a metric space of
Hausdorff dimension $\alpha^{\,\prime}.$ Assume that, $X^{\,\prime}$
is Ahlfors $\alpha^{\,\prime}$-regular, path connected, locally
connected, proper and Ptolemaic metric space, which supports $(1;
p)$-Poincar\'{e} inequality. Suppose that there exists
$\varepsilon_0>0$ and a Lebesgue measurable function
$\psi(t)\colon(0, \varepsilon_0)\rightarrow [0,\infty]$ with the
following property: for every $\varepsilon_2\in(0, \varepsilon_0]$
there is $\varepsilon_1\in (0, \varepsilon_2]$ such that
(\ref{eq5C}) holds for every $\varepsilon\in (0,\varepsilon_1).$
Assume also that (\ref{eq4*}) holds as $\varepsilon\rightarrow 0.$

Let $K$ be some nondegenerate continuum in $X.$ If an open, discrete
ring $Q$-mapping $f\colon D\setminus\{\zeta_0\}\rightarrow
X\setminus K$ at $\zeta_0$ with respect $(p, q)$-moduli satisfies
the condition $\textbf{A},$ then $f$ has a continuous extension at
$\zeta_0.$ (Here the existing of limit at $\zeta_0$ is understood in
the sense of the space $(\overline{X}, h_{x_0})$).
}
 \end{lemma}

\medskip
\begin{proof}
Since $X$ is locally compact, we may consider that
$\overline{B(\zeta_0, \varepsilon_0)}$ is a compact. Suppose the contrary, i.e., suppose that $f$ has no limit at
$\zeta_0.$ By Lemma \ref{lem5}, $(\overline{X}, h_{x_0})$ is a compact, therefore,
$C(f, \zeta_0)$ is non-empty. Thus, there exist two sequences $x_j$ and $x_j^{\,\prime}$ in $B(\zeta_0,
\varepsilon_0)\setminus\left\{\zeta_0\right\},$ $d(x_j,
\zeta_0)\rightarrow 0,\quad d(x_j^{\,\prime}, \zeta_0)\rightarrow 0$
such that $h_{x_0}\left(f(x_j),\,f(x_j^{\,\prime})\right)\geqslant
a>0$ for all $j\in {\Bbb N}.$ Since $G$ is locally path connected at $\zeta_0,$
there exists a sequence
$r_k\rightarrow 0,$ $0<r_k<\varepsilon_0,$ $r_1>r_2>r_3>\ldots,$
such that $B(\zeta_0, r_k)\subset V_k\subset B(\zeta_0, r_{k-1})$ and
$V_k\cap G=V_k\setminus \{\zeta_0\}$ is path connected set.
Since $d(x_j, \zeta_0)\rightarrow 0$ and $d(x_j^{\,\prime},
\zeta_0)\rightarrow 0$ as $j\rightarrow\infty,$ there is a number
$j_1\in {\Bbb N}$ such that $x_{j_1}$ and $x_{j_1}^{\,\prime}\in
B(\zeta_0, r_2).$ Let $C_{j_1}$ be a curve joining $x_{j_1}$ and $x_{j_1}^{\,\prime}$
in $V_2\setminus \{\zeta_0\}\subset B(\zeta_0,
r_1)\setminus \{\zeta_0\}.$ Similarly, there is a number $j_2\in {\Bbb N}$
such that $x_{j_2}$ and $x_{j_2}^{\,\prime}\in B(\zeta_0, r_3).$ Let $C_{j_2}$ be a curve
joining $x_{j_2}$ and $x_{j_2}^{\,\prime}$ in
$V_3\setminus \{\zeta_0\}\subset B(\zeta_0, r_2)\setminus
\{\zeta_0\}.$ Continuing this process, we obtain some number $j_k\in
{\Bbb N}$ such that $x_{j_k}$ and $x_{j_k}^{\,\prime}\in B(\zeta_0,
r_{k+1}).$ We join $x_{j_k}$ and $x_{j_k}^{\,\prime}$ by a curve
$C_{j_k},$ which belongs to $V_{k+1}\setminus \{\zeta_0\}\subset B(\zeta_0,
r_k)\setminus \{\zeta_0\}.$ There is no loss of generality in assuming that $x_j$ and $x_j^{\,\prime}$
can be joined by the curve $C_j$ in $\overline{B(\zeta_0,
r_j)}\setminus\left\{\zeta_0\right\}.$

Let $E_j=(B(\zeta_0, \varepsilon_0)\setminus \{\zeta_0\}, C_j),$ and
let $\Gamma_{f(E_j)}$ be a family of curves, which corresponds to a
condenser $f(E_j)$ in (\ref{eq28}). Since ${\rm cap}_p f(E_j)={\rm
cap}_p(f(B(\zeta_0, \varepsilon_0)\setminus \{\zeta_0\}),
f(C_j))\geqslant {\rm cap}_p(X\setminus K, f(C_j)),$ by Lemma
\ref{lem2} we obtain that $\Gamma_{f(E_j)}\ne \varnothing.$ Let
$\Gamma_j^{\,*}$ be a family of all maximal $f$-liftings of
$\Gamma_{f(E_j)}$ in $B(\zeta_0,
\varepsilon_0)\setminus\left\{\zeta_0\right\}$ starting at $C_j.$
This family of curves is well defined, because $f$ satisfies the
condition $\textbf{A}$ by assumption of Lemma.

Let $\Gamma_{E_{j_1}}$ be a family of all curves
$\alpha(t)\colon[a,\,c)\rightarrow B(\zeta_0,
\varepsilon_0)\setminus\left\{\zeta_0\right\}$ starting at $C_j,$
for which $\alpha(t_k)\rightarrow \zeta_0$ at some sequence $t_k\rightarrow c-0,$ $t_k\in [a,\,c),$
$k\rightarrow\infty.$ Similarly, let $\Gamma_{E_{j_2}}$ be a family of all curves
$\alpha(t)\colon[a,\,c)\rightarrow B(\zeta_0,
\varepsilon_0)\setminus\left\{\zeta_0\right\}$ starting at $C_j,$ for which ${\rm dist}\left(\alpha(t_k),
S(\zeta_0, \varepsilon_0)\right)\rightarrow 0$ for some sequence $t_k\rightarrow c-0,$ $t_k\in [a,\,c),$
$k\rightarrow\infty.$
Now, we show that
\begin{equation}\label{eq33*!}
\Gamma_j^{\,*}\,=\,\Gamma_{E_{j_1}}\cup \Gamma_{E_{j_2}}\,,
\end{equation}

\medskip
Suppose the contrary, i.e., suppose that there exists a curve $\beta\colon
[a,\,b)\rightarrow X^{\,\prime}$ in the
family $\Gamma_{f(E_j)}$ such that its maximal lifting $\alpha\colon
[a,\,c)\rightarrow B(\zeta_0, \varepsilon_0)\setminus\{\zeta_0\}$
satisfies the condition $d(|\alpha|, S(\zeta_0,
\varepsilon_0)\cup\{\zeta_0\})=\delta_0>0.$ Consider the set
$$P=\left\{x\in X:\, x=\lim\limits_{k\rightarrow\,\infty}
\alpha(t_k)
 \right\}\,,\quad t_k\,\in\,[a,\,c)\,,\quad
 \lim\limits_{k\rightarrow\infty}t_k=c\,,$$
where $\lim$ is understood with respect to the metric $d.$ First, we
observe that $c\ne b,$ because otherwise $|\beta|=f(|\alpha|)$ is a
compact subset of $f(B(\zeta_0,
\varepsilon_0)\setminus\{\zeta_0\}),$ which contradicts the choice
of $\beta.$

\medskip
So, $c\ne b,$ and, passing to subsequences if necessary, we can
restrict ourselves to monotone sequences $t_k.$ If $x\in P,$ by the
continuity of $f$ we see that
$f\left(\alpha(t_k)\right)\stackrel{d^{\,\prime}}{\rightarrow}\,f(x)$
as $k\rightarrow\infty,$ where $t_k\in[a,\,c),\,t_k\rightarrow c$ as
$k\rightarrow \infty.$ However,
$f\left(\alpha(t_k)\right)=\beta(t_k)\stackrel{d^{\,\prime}}{\rightarrow}\beta(c)$
as $k\rightarrow\infty.$ Thus, $f$ is a constant on $P$ in
$B(\zeta_0, \varepsilon_0)\setminus\{\zeta_0\}.$ On other hand,
$\overline{|\alpha|}$ is a compact closed subset of the compact set
$\overline{B(\zeta_0, \varepsilon_0)}$ (see \cite[Theorem~2.II.4, $\S\,41$]{Ku}).
The Cantor
condition for the compact set $\overline{|\alpha}|$ shows that
$$P\,=\,\bigcap\limits_{k\,=\,1}^{\infty}\,\overline{\alpha\left(\left[t_k,\,c\right)\right)}
\ne\varnothing\,,
$$
%
because the sequence
$\alpha\left(\left[t_k,\,c\right)\right)$ of connected sets is monotone;
see \cite[1.II.4, $\S\,41$]{Ku}. By \cite[Theorem~5.II.5, $\S\,47$]{Ku}, the set
$P$ is connected. Since the mapping $f$ is discrete, $P$ is a
singleton. Thus, the curve $\alpha\colon [a,\,c)\rightarrow\,B(\zeta_0,
\varepsilon_0)\setminus\{\zeta_0\}$ can be extended to a closed curve
$\alpha\colon [a,\,c]\rightarrow B(\zeta_0,
\varepsilon_0)\setminus\{\zeta_0\},$ moreover, $f\left(\alpha(c)\right)=\beta(c).$ By condition $\textbf{A}$
there exists yet another maximal lifting $\alpha^{\,\prime}$ with origin at $\alpha(c).$
Uniting the liftings $\alpha$ and $\alpha^{\,\prime},$
we obtain a new lifting $\alpha^{\,\prime\prime}$ for $\beta,$ defined on
$[a,
c^{\prime}),$ \,\,$c^{\,\prime}\,\in\,(c,\,b).$ This contradicts the
maximality of the initial lifting $\alpha.$ Thus, $d(|\alpha(t)|, S(\zeta_0,
\varepsilon_0)\cup\{\zeta_0\})\rightarrow 0$ as $t\rightarrow c-0.$

\medskip
By~\eqref{eq33*!}, we obtain that
\begin{equation}\label{eq34*!}
M_p\left(\Gamma_{f(E_j)}\right)\leqslant
M_p(f(\Gamma_{E_{j_1}}))\,+\,M_p(f(\Gamma_{E_{j_2}}))\,.
\end{equation}
By Lemma~\ref{lem3.1!} $M_p(f(\Gamma_{E_{j_1}}))=0.$

Note that an arbitrary curve $\gamma\in \Gamma_{E_{j_2}}$ is not included
entirely both in $B(\zeta_0, \varepsilon_0-\frac{1}{m})$ and $X\setminus
B(\zeta_0, \varepsilon_0-\frac{1}{m})$ for sufficiently large $m.$ Thus, there exists
$y_1\in |\gamma|\cap S(\zeta_0, \varepsilon_0-\frac{1}{m})$ (see \cite[Theorem 1, $\S\,$46, item
I]{Ku}). Let $\gamma:[0,
1]\rightarrow X$ and let $t_1\in (0, 1)$ be such that $\gamma(t_1)=y_1.$
There is no loss of generality in
assuming that $|\gamma|_{[0, t_1)}|\subset
B(\zeta_0, \varepsilon_0-1/m).$ We put $\gamma_1:=\gamma|_{[0,
t_1)}.$ Observe that $|\gamma_1|\subset B(\zeta_0, \varepsilon_0-1/m),$
moreover, $\gamma_1$ is not included entirely either in $\overline{B(\zeta_0,
r_j)}$ or in $X\setminus\overline{B(\zeta_0, r_j)}.$ Consequently,
there exists $t_2\in (0, t_1)$ with $\gamma_1(t_2)\in S(\zeta_0, r_j)$ (see \cite[Theorem 1, $\S\,$46, item
I]{Ku}). There is no loss of generality in
assuming that $|\gamma_{[t_2, t_1]}|\subset
X\setminus\overline{B(\zeta_0, r_j)}.$ Put
$\gamma_2=\gamma_1|_{[t_2, t_1]}.$ Observe that $\gamma_2$ is a subcurve of $\gamma.$ By the said above,
$\Gamma_{E_{j_2}}>\Gamma(S(\zeta_0, r_j), S(\zeta_0,
\varepsilon_0-\frac{1}{m}), A(\zeta_0, r_j,
\varepsilon_0-\frac{1}{m}))$ for sufficiently large $m\in {\Bbb
N}.$ Set
$A_{j}=\{x\in X: r_j<d(x, \zeta_0)< \varepsilon_0-\frac{1}{m}\}$ and
$$ \eta_{j}(t)= \left\{
\begin{array}{rr}
\psi(t)/I(r_j, \varepsilon_0-\frac{1}{m}), &   t\in (r_j,\, \varepsilon_0-\frac{1}{m}),\\
0,  &  t\in {\Bbb R} \setminus (r_j,\, \varepsilon_0-\frac{1}{m}).
\end{array}
\right.$$
Observe that
$\int\limits_{r_j}^{\varepsilon_0-\frac{1}{m}}\,\eta_{j}(t) dt
=\,\frac{1}{I\left(r_j,
\varepsilon_0-\frac{1}{m}\right)}\int\limits_{r_j}
^{\varepsilon_0-\frac{1}{m}}\,\psi(t)dt=1. $ By the definition of
ring $Q$-mapping at $\zeta_0$ with respect to $(p, q)$-moduli and
by~\eqref{eq34*!}, we obtain that
$$M_p(f(\Gamma_{E_{j}}))\leqslant\,\frac{1}{{I^q(r_j,
\varepsilon_0-\frac{1}{m})}}\int\limits_{r_j<d(x,
\zeta_0)<\varepsilon_0}\,Q(x)\,\psi^{\,q}(d(x, \zeta_0))\,d\mu(x)\,.
$$
Passing to the limit as $m\rightarrow\infty,$ we obtain
$$
M_p(f(\Gamma_{E_{j}}))\leqslant\, {\mathcal
S}(r_j):=\frac{1}{{I^q(r_j, \varepsilon_0)}}\int\limits_{r_j<d(x,
\zeta_0)<\varepsilon_0}\,Q(x)\,\psi^{\,q}(d(x, \zeta_0))\,d\mu(x).
$$
%
Formula~\eqref{eq4*} shows that ${\mathcal S}(r_j)\,\rightarrow\, 0$
as $j\rightarrow \infty$ and, consequently, from~\eqref{eq34*!}
it follows that
\begin{equation}\label{eq3D}
M_p\left(\Gamma_{f(E_j)}\right)\rightarrow 0\,,\qquad
j\rightarrow\infty\,.
\end{equation}
On the other hand, by Lemma~\ref{lem2}, ${\rm cap}_p
f(E_j)=M_p\left(\Gamma_{f(E_j)}\right)\geqslant \delta>0$ for every
$j\in {\Bbb N}.$ But this conclusion contradicts (\ref{eq3D}). Thus,
$f$ has a limit at $\zeta_0,$ as required.~\end{proof}

\medskip
\section{Proof of the main result}

We will say that a space  $(X,d,\mu)$ is {\it upper $\alpha$-regular
at a point} $x_0\in X$ if there is a constant $C> 0$ such that
$$
\mu(B(x_0,r))\leqslant Cr^{\alpha}$$
for the balls $B(x_0,r)$ centered at $x_0\in X$ with all radii
$r<r_0$ for some $r_0>0.$ We will also say that a space  $(X,d,\mu)$
is {\it upper $\alpha$-regular} if the above condition holds at
every point $x_0\in X.$ The following statement can be found in
\cite[Lemma~13.2]{MRSY}.

\medskip
\begin{proposition}\label{pr3}
{\sl Let $G$ be a domain Ahlfors $\alpha$-regular metric space $(X,
d, \mu)$ at $\alpha\geqslant 2.$ Assume that $x_0\in \overline{G}$
and $Q:G\rightarrow [0, \infty]$ belongs to $FMO(x_0).$ If
%
$$\mu(G\cap B(x_0, 2r))\leqslant
\gamma\cdot\log^{\alpha-2}\frac{1}{r}\cdot \mu(G\cap B(x_0, r))$$
%
for some $r_0>0$ and every $r\in (0, r_0),$ then $Q$
satisfies
$$
\int\limits_{\varepsilon<d(x, x_0)<\varepsilon_0}
Q(x)\cdot\psi_{\varepsilon}^{\alpha}(d(x, x_0))\, d\mu(x)\leqslant
F(\varepsilon, \varepsilon_0)\qquad\forall\,\,\varepsilon\in(0,
\varepsilon_0^{\,\prime})\,,
$$
where $G(\varepsilon):=F(\varepsilon,
\varepsilon_0)/I^{\alpha}(\varepsilon, \varepsilon_0),$
$I(\varepsilon,
\varepsilon_0):=\int\limits_{\varepsilon}^{\varepsilon_0}\psi(t)dt$
and $\psi(t):=\frac{1}{t\log\frac{1}{t}}.$}
\end{proposition}

\medskip
The following main result of the paper follows from Lemma~\ref{lem4*}
(see also similar result in \cite{Sev$_1$} for the space ${\Bbb R}^n$).

\medskip
\begin{theorem}\label{th1}{\sl\, Let $2\leqslant\alpha, \alpha^{\,\prime}<\infty,$ let
$\alpha^{\,\prime}-1<p\leqslant\alpha$ and $1\leqslant q\leqslant
\alpha.$
Let $G:=D\setminus\{\zeta_0\}$ be a domain in a locally compact
metric space $(X, d, \mu)$ of Hausdorff dimension $\alpha,$ where
$G$ is locally path connected at $\zeta_0\in D,$ and let
$(X^{\,\prime}, d^{\,\prime}, \mu^{\,\prime})$ be a metric space of
Hausdorff dimension $\alpha^{\,\prime}.$ Assume that, $X^{\,\prime}$
is Ahlfors $\alpha^{\,\prime}$-regular, path connected, locally
connected, proper and Ptolemaic metric space, which supports $(1;
p)$-Poincar\'{e} inequality. Suppose that $Q\in FMO(\zeta_0).$

Let $K$ be some nondegenerate continuum in $X.$ If an open, discrete
ring $Q$-mapping $f\colon D\setminus\{\zeta_0\}\rightarrow
X\setminus K$ at $\zeta_0$ with respect to $(p, q)$-moduli satisfies
the condition $\textbf{A},$ then $f$ has a continuous extension at
$\zeta_0.$ (Here the existing of limit at $\zeta_0$ is understood in
the sense of the space $(\overline{X}, h_{x_0})$).
}
 \end{theorem}

\medskip
\begin{proof} We show that the condition $Q\in FMO(\zeta_0)$ implies the conditions (\ref{eq5C})--(\ref{eq4*})
at $\zeta_0.$ In fact, putting
$\psi(t)=\log^{-\alpha/q}\frac{1}{t},$ we obtain the relations
(\ref{eq5C})--(\ref{eq4*}) from Proposition \ref{pr3}. Now we obtain
the desired conclusion by Lemma~\ref{lem4*}.
\end{proof}~$\Box$

\medskip
{\it Proof of the Theorem \ref{th2}.}
Assume the contrary, i.e., assume that there exists $A\in X^{\,\prime}$ with
\begin{equation}\label{eq30}
d^{\,\prime}(f(x), A)\geqslant \delta_0
\end{equation}
for every $x\in B(\zeta_0, \varepsilon_0)\setminus\{\zeta_0\}$ and
some $\varepsilon_0>0.$ By (\ref{eq30}), $f(x)\in
X^{\,\prime}\setminus B(A, \delta_0)$ for $x\in B(\zeta_0,
\varepsilon_0)\setminus\{\zeta_0\}.$ Since $X^{\,\prime}$ is proper,
$X^{\,\prime}$ is locally compact. Since $X^{\,\prime}$ is locally
connected and locally compact space, there exists $t_1>0,$
$t_1<\delta_0,$ and a continuum $V$ such that $B(A, t_1)\subset
V\subset B(A, \delta_0).$ Since $X^{\,\prime}$ is Ahlfors regular,
$B(A, t_1)$ does not degenerate into a point. Now, $V$ is
non-degenerate continuum. Moreover, since $V\subset B(A, \delta_0),$
it follows from (\ref{eq30}) that $f$ does not take values in $V.$
By Theorem~\ref{th1}, $f$ has isolated singularity as
$x\rightarrow\zeta_0,$ that contradicts to assumption of the theorem.~$\Box$

\medskip
\medskip
{\bf \noindent Evgeny Sevost'yanov, Antonina Markysh} \\
Zhytomyr Ivan Franko State University,  \\
40 Bol'shaya Berdichevskaya Str., 10 008  Zhytomyr, UKRAINE \\
Phone: +38 -- (066) -- 959 50 34, \\
Email: esevostyanov2009@gmail.com, tonya@bible.com.ua
\end{document}